\documentclass[twoside]{article}
\usepackage{a4wide, amsmath, amsthm, amssymb, amscd, mathptmx, mathtools}
\usepackage{hyperref}
\usepackage{enumitem}
\usepackage{braket} 
\usepackage{soul,xcolor} 
\setstcolor{red} 
\usepackage[colorinlistoftodos, textwidth=2.7cm]{todonotes}
\setuptodonotes{fancyline, color=orange}
\usepackage{setspace} 
\usepackage{amsfonts}
\usepackage{bbm}
\usepackage{tikz}

\usetikzlibrary{arrows}

\usepackage{scalerel,stackengine}
\stackMath
\newcommand\reallywidehat[1]{%
	\savestack{\tmpbox}{\stretchto{%
			\scaleto{%
				\scalerel*[\widthof{\ensuremath{#1}}]{\kern-.6pt\bigwedge\kern-.6pt}%
				{\rule[-\textheight/2]{1ex}{\textheight}}
			}{\textheight}%
		}{0.5ex}}%
	\stackon[1pt]{#1}{\tmpbox}%
}
\newcommand\imCMsym[4][\mathord]{%
	\DeclareFontFamily{U} {#2}{}
	\DeclareFontShape{U}{#2}{m}{n}{
		<-6> #25
		<6-7> #26
		<7-8> #27
		<8-9> #28
		<9-10> #29
		<10-12> #210
		<12-> #212}{}
	\DeclareSymbolFont{CM#2} {U} {#2}{m}{n}
	\DeclareMathSymbol{#4}{#1}{CM#2}{#3}
}

\imCMsym{cmmi}{124}{\Cjmath}

\usepackage{color}
\usepackage{pstricks}

\usepackage[all]{xy}\CompileMatrices\SelectTips{cm}{12}

 \makeatletter 
 \providecommand\@dotsep{5} 
 \makeatother 
 
\theoremstyle{plain}
\newtheorem{Thm}{\sc Theorem}[section]
\newtheorem{Theorem}[Thm]{\sc Theorem}
\newtheorem{Corollary}[Thm]{\sc Corollary}

\newtheorem*{Corollary*}{\sc Corollary}

\newtheorem{Proposition}[Thm]{\sc Proposition}
\newtheorem*{Proposition*}{\sc Proposition}
\newtheorem{Lemma}[Thm]{\sc Lemma}

\theoremstyle{definition}
\newtheorem{Definition}[Thm]{Definition}

\theoremstyle{remark}
\newtheorem{Remark}[Thm]{Remark}
\newtheorem{Example}[Thm]{Example}
\newtheorem*{Example*}{Example}
\newtheorem*{Remark*}{Remark}

\renewcommand{\AA}{{\mathbb A}}
\newcommand{\CC}{{\mathbb C}}

\newcommand{\EE}{{\mathbb E}}
\newcommand{\FF}{{\mathbb F}}
\newcommand{\GG}{{\mathbb G}}

\newcommand{\ZZ}{{\mathbb Z}}

\newcommand{\cE}{{\mathcal E}}

\newcommand{\cO}{{\mathcal O}}

\newcommand{\cU}{{\mathcal U}}

\newcommand{\Pic}{{\mathop{\rm Pic \, }}}

\newcommand{\Coh}{{\mathop{\operatorname{Coh}\, }}}

\newcommand{\Hom}{{\mathop{{\rm Hom}}}}

\newcommand{\id}{\mathop{\rm Id}}

\newcommand{\Ext}{{\mathop{{\rm Ext \,}}}}

\newcommand{\Spec}{{\mathop{{\rm Spec\, }}}}

\newcommand{\Vect}{{\mathop{{\rm Vect \,}}}}

\def \Coh {{\sf Coh}}

\def \perf {{\sf perf}}
\def \Vect {{\sf Vect}}

\newcommand\Fdiv {{F\sf{\textendash div}}}
\newcommand{\unit}{{\mathbbm{ 1}}}

\begin{document}
	
	\markboth {$F$-divided bundles}{$F$-divided bundles}

	\title{Relative Gieseker's problem on $F$-divided bundles}
	\author{Adrian Langer}
	\date{
        \today
    }
	
	\maketitle
	

	{\noindent \sc Address:}\\
	Institute of Mathematics, University of Warsaw,
	ul.\ Banacha 2, 02-097 Warszawa, Poland\\
	e-mail: {\tt alan@mimuw.edu.pl}

\begin{abstract}
Let $f: X\to Y$ be a proper surjective morphism of varieties defined over an algebraically closed field of positive characteristic. We prove that if $f$ has geometrically connected fibers then the induced homomorphism of $F$-divided fundamental groups is faithfully flat.  An important new ingredient in our proof is an analogue of B. Bhatt's and P. Scholze's descent theorem \cite[Theorem 1.3]{Bhatt-Scholze2017} for $F$-divided bundles.

As a corollary, we prove that in general if $X$ is normal, $Y$ is smooth, both $X$ and $Y$ are projective, and the induced map on \'etale fundamental groups is surjective,  then the corresponding homomorphism on $F$-divided fundamental groups is faithfully flat. We also establish an analogous result for isomorphisms. This generalizes and strengthens a recent result of X. Sun and L. Zhang \cite{Sun-Zhang2025}, which in turn generalized earlier results of H. Esnault and V. Mehta \cite{Esnault-Mehta2010} and I. Biswas, M. Kumar, and A. J. Parameswaran \cite{Biswas-Parameswaran-Kumar2025}.
\end{abstract}

\section*{Introduction}

Let $X$ be a connected scheme of finite type over some algebraically closed field $k$.
In \cite[Appendix]{Grothendieck1968} A. Grothendieck defined a \emph{coherent stratified sheaf} on $X/k$ as
a coherent sheaf $E$ on $X$ together with an isomorphism $\varphi:  p_1^*E \xrightarrow{\simeq} p_2^*E$ on the formal completion $\widehat{X\times _kX}$ of $X\times _kX$ along the diagonal, subject to the cocycle condition $p_{23}^*(\varphi) \circ p_{12}^*(\varphi) = p_{13}^*(\varphi)$ on the formal completion \( \reallywidehat{X \times_k X \times_k X}\) of $X\times _k X\times _k X$ along the diagonal. Such stratified sheaves can also be viewed as {crystals}  on the infinitesimal site $(X/k)_{\sf inf}$. In a later work \cite{Grothendieck1970},  Grothendieck used stratified sheaves  to interpret the Riemann-Hilbert correspondence for proper complex varieties. Namely, if $X$ is a proper complex variety then the category of representations of the topological fundamental group $\pi _1^{\sf top} (X^{\sf an})$ of the analytification of $X$ is equivalent to the category $\Coh^{\sf strat}(X)$ of coherent stratified sheaves on $X/\CC$ (see \cite[(4.1.5)]{Grothendieck1970}). 
In particular,  $\Coh^{\sf strat}(X)$ is equivalent  to the category  $\Vect^{\sf strat}(X)$ of stratified vector bundles on $X/\CC$. Later N. Saavedra Rivano in \cite[Chapitre VI, 1.2]{Saavedra72} proved that in general the category $\Vect^{\sf strat}(X)$ is  Tannakian and  hence when fixing a $k$-point of $X$, Tannakian duality gives us a stratified fundamental group $ \pi^{\sf strat}_1(X)$. This is a purely algebraic object, and for complex varieties its representations are equivalent to that of $\pi _1^{\sf top} (X^{\sf an})$. Using this language we can then reformulate (and slightly strengthen) \cite[Theoreme 4.2]{Grothendieck1970} as follows:

\begin{Theorem}\label{main1}
Let $f: X\to Y$ be a morphism of proper connected schemes over $k=\CC$.  
\begin{enumerate}
\item 
	If  $f_*: \pi^{\sf{\acute et}}_1(X)\to\pi^{\sf{\acute et}}_1(Y) $ is surjective then
	$f_*: \pi^{\sf strat}_1(X)\to\pi^{\sf{strat}}_1(Y) $  is faithfully flat.
	\item  If 	$f_*: \pi^{\sf{\acute et}}_1(X)\to\pi^{\sf{\acute et}}_1(Y) $ is an isomorphism then
	$f_*: \pi^{\sf strat}_1(X)\to\pi^{\sf{strat}}_1(Y) $ is an isomorphism.
\end{enumerate}
\end{Theorem}

The result above is often stated only for smooth complex varieties, where a stratified vector bundle is equivalent to a vector bundle equipped with an integrable connection, or, equivalently, a $D$-module. Although all the objects involved are algebraic, the proof of Theorem \ref{main1} relies on analytic methods. In particular, it depends on the version of the Riemann--Hilbert correspondence mentioned earlier, as well as on the Grothendieck--Malcev theorem, which relates representations of a finitely generated group to its profinite completion. 

\medskip

If $k$ has positive characteristic then the category  $\Vect^{\sf strat}(X)$ of stratified vector bundles on $X/k$ is equivalent to the category  $\Vect^\perf(X)$ of so called $F$-divided bundles (see \cite[Proposition 3.4]{Esnault-Srinivas2016}).  An \emph{$F$-divided vector bundle} on	$X$ is a sequence $\EE:=\{E_i, \sigma_i\} _{i\in \ZZ_{\ge 0}}$ of vector bundles  $E_i$ on $X$ and $\cO_X$-isomorphisms $\sigma _i:F_{X}^*E_{i+1}\xrightarrow{\simeq} E_i$. Again if $X$ is smooth then giving a stratified vector bundle is equivalent to giving a $D$-module (see Katz's \cite[Theorem 1.3]{Gieseker1975}). However, this condition is significantly stronger than merely specifying a vector bundle with an integrable connection.

As in the characteristic zero case, $\Vect^\perf(X)$ is a Tannakian  category and when fixing a $k$-point, Tannakian duality gives rise to an $F$-divided fundamental group $ \pi^{F\sf{\textendash div}}_1(X)$.
Then we have the following analogue of Grothendieck's Theorem \ref{main1}
(see Theorem \ref{faithful-flatness-for-morphisms} and Corollary \ref{cor:isomorphism-for-morphisms}).

\begin{Theorem} \label{main3}
Let $f: X\to Y$ be a surjective morphism of normal projective varieties defined over an algebraically closed field of positive characteristic.  Assume that $Y$ is smooth.
\begin{enumerate}
	\item 
	If  $f_*: \pi^{\sf{\acute et}}_1(X)\to\pi^{\sf{\acute et}}_1(Y) $ is surjective then
	$f_*: \pi^{F\sf{\textendash	div}}_1(X)\to\pi^{F\sf{\textendash div}}_1(Y) $ is faithfully flat.
	\item  If $f_*: \pi^{\sf{\acute et}}_1(X)\to\pi^{\sf{\acute et}}_1(Y) $ is an isomorphism then
	$f_*: \pi^{F\sf{\textendash	div}}_1(X)\to\pi^{F\sf{\textendash div}}_1(Y) $ is an isomorphism.
\end{enumerate}
\end{Theorem}

The absolute version of the above theorem when $X$ is smooth and $Y$ is a point was proven in \cite{Esnault-Mehta2010} 
in response to a problem posed by Gieseker  (see \cite[\textsection 2]{Gieseker1975}). Recently, this result was generalized to normal proper varieties  in \cite{Langer-Zhang2025}.  Theorem \ref{main3} positively answers \cite[Question 3.2 (i)]{Esnault2013}  in the case of a morphism between from normal to smooth projective variety. Note that for our proof it is essential to consider normal varieties even if one is interested only in smooth varieties as in the original formulation of the question. 

The first part of the above theorem was previously unknown. The second part, in the case where both $X$ and $Y$ are smooth, appears in \cite{Sun-Zhang2025}, who provide a different and more involved proof.  In the special case where $f$ is a finite genuinely ramified morphism, the second part of the above theorem was proven for smooth varieties in \cite{Biswas-Parameswaran-Kumar2025} and this result plays a crucial role in our proof. This is the only reason why we add the assumption on smoothness of $Y$. Unfortunately, in the proofs of both \cite[Proposition 4.1]{Biswas-Parameswaran-Kumar2025} and \cite[Lemma 3.2]{Sun-Zhang2025}, the smoothness of $Y$ is essential, and the behaviour of \'etale fundamental groups under spreading out of a finite morphism between general normal varieties is not well understood. However, this assumption is not needed if $f$ has geometrically connected fibers:

\begin{Theorem}\label{main4}
Let $f: X\to Y$ be a proper surjective morphism of varieties defined over an algebraically closed field of positive characteristic.  Assume that  $f$ has geometrically connected fibers. Then $f_*: \pi^{F\sf{\textendash	div}}_1(X)\to\pi^{F\sf{\textendash div}}_1(Y) $ is faithfully flat. Moreover, if $X$ and $Y$ are projective and $f_*: \pi^{\sf{\acute et}}_1(X)\to\pi^{\sf{\acute et}}_1(Y) $ is injective then
$f_*: \pi^{F\sf{\textendash	div}}_1(X)\to\pi^{F\sf{\textendash div}}_1(Y) $ is an isomorphism.
\end{Theorem}

The first part of the above theorem follows from more general Corollary \ref{proper-fibration-relative-gerbe}. It is an analogue of  \cite[Expos\'e IX, Corollaire 5.6]{SGA1} for the $F$-divided fundamental group. This result was known only in case $f$ is a smooth morphism between smooth projective varieties (see \cite[Theorem 1.1]{dos-Santos2015}).
The second part follows from  Theorem \ref{isomorphism-for-morphisms} and Lemma \ref{passing-to-universal-homeomorphism}.

\medskip

One of the main ingredients in our proof of Theorem \ref{main3} (see Theorem \ref{descent-for-proper-fibrations-old})
and Theorem \ref{main4} is the following analogue of the descent theorem \cite[Theorem 1.3]{Bhatt-Scholze2017} of B. Bhatt and P. Scholze in the case of $F$-divided bundles, which might be of independent interest. 

\begin{Theorem} \label{main2}
	Let $f: X\to Y$ be a proper surjective map of connected Noetherian $F$-finite $\FF_p$-schemes. Assume that all geometric fibers of $f$ are connected.  
	Then  an $F$-divided bundle $\EE$ on $X$ descends (necessarily uniquely) to an $F$-divided bundle on $Y$  if and only if
	$\EE$ is trivial on all geometric fibers of $f$. 
\end{Theorem}

In fact, \cite[Theorem 1.3]{Bhatt-Scholze2017} already implies that we can descend each $E_i$ separately after some finite purely inseparable map.  Since such maps are dominated by some iterates of the absolute Frobenius,
there exists some integer $e\ge 0$ such that $(F_X^{e})^*E_0$ descends to $Y$. Then, for all $i\ge 0$,  the bundles $(F_X^{e+i})^*E_i$ descend to $Y$, so \cite[Theorem 1.3]{Bhatt-Scholze2017} does not give us any additional information. In particular, it does not allow us to descend $\EE$ to $Y$. The main challenge in our proof is to find a single map that works for all $E_i$ at the same time and notably, our argument does not rely on the result of Bhatt and Scholze.

\medskip

The proof of the first part of Theorem \ref{main3} uses just Theorem \ref{main2} and various results on genuinely ramified morphism. But as in all the previous papers on the subject, the proof of the second part  of Theorem \ref{main3}  depends on important Hrushovski's theorem \cite[Corollary 1.2]{Hrushovski2004} and boundedness of families of semistable sheaves with fixed numerical invariants (see \cite[Theorem 4.4]{Langer2004}). In this paper, this last result is hidden in X. Sun's generalization of Simpson's representation spaces to positive characteristic (see \cite[Theorem 2.3]{Sun2019}), which is also an important ingredient of the proof. As in \cite{Esnault-Srinivas2019} one
could also try to prove Theorem \ref{main3}, (2) using moduli spaces of semistable vector bundles constructed in \cite{Langer2004}. Then it is easy to descend simple objects of $\Vect^\perf(X)$ but as in \cite[Section 5]{Esnault-Srinivas2019} this approach runs into technical difficulties involving study of universal extensions on products of moduli spaces making it impractical due to a horrible notation and complications.

\medskip

The structure of the paper is as follows. In the first section, we recall some facts on $F$-divided bundles and representation schemes. In Section 2 we prove Theorem \ref{main2} and the first part of Theorem \ref{main4} as its simple corollary. Section 3 contains the proofs of Theorem \ref{main1}, Theorem \ref{main3} and the remaining part of Theorem \ref{main4}.

\section{Preliminaries}

\subsection{$F$-divided bundles}

Let $X$ be an $\FF_p$-scheme. An \emph{$F$-divided vector bundle} on	$X$ is a sequence $\{E_i, \sigma_i\} _{i\in \ZZ_{\ge 0}}$ of vector bundles  $E_i$ on $X$ and $\cO_X$-isomorphisms $\sigma _i:F_{X}^*E_{i+1}\to E_i$.
The category of $F$-divided bundles on $X$ is denoted by $\Vect^\perf(X)$. It comes with the \emph{unit} $\unit_X$ defined by the constant sequence $\{\cO_X\}_{i\in \ZZ_{\ge 0}}$ with canonical isomorphisms $F_{X}^*\cO_X\simeq \cO_X$.

Below we recall the following result from \cite{Langer-Zhang2025} that generalizes J.-P.-S. dos Santos's result from \cite{dos_Santos2007} (see also \cite[Proposition 3.3]{Esnault-Srinivas2016}).

\begin{Theorem}\label{F-divided-sheaves-form-Tannakian-category}
Let $X$ be an $F$-finite Noetherian connected scheme. Then  $\Vect^\perf(X)$ is a Tannakian category over
the inverse limit perfection $\cO_X(X)^{\sf perf}$ of the ring $\cO_X(X)$. The corresponding Tannakian gerbe $\Pi_X^{\Fdiv}$ is pro-smooth
banded.
\end{Theorem}

We will also need the following lemma from \cite[Lemma 5.9]{Langer-Zhang2025}.
	
\begin{Lemma}\label{evaulation-map}
Let  $X$ be a connected Noetherian  $F$-finite $\FF_p$-scheme and let $k=\cO_X(X)^{\sf perf}$. Then 
for any $\EE\in \Vect^{\perf}(X)$,	$\EE ^h(X):=\Hom (\unit_X, \EE)$ is a finite dimensional $k$-vector space and the canonical evaluation map $\EE^h(X)\otimes _k \unit_X\to \EE$
is injective.
\end{Lemma}

In addition, we require the following simple result (see \cite[Proposition 3.3]{Esnault-Srinivas2016} and \cite[Lemma 2.3]{Langer-Zhang2025}).

\begin{Lemma}\label{passing-to-universal-homeomorphism}
	Let $f\colon Y\to X$ be a  finite universal homeomorphism of Noetherian $\FF_p$-schemes. Then the pullback gives rise to an
	equivalence of categories ${\sf Vect} ^{\sf perf}(X)\simeq {\sf Vect} ^{\sf perf}(Y)$. 
\end{Lemma}

A direct corollary of the above lemma is the following:

\begin{Lemma} \label{lem:finiteness}
	If $X\subset Y$ is a closed immersion of Noetherian $\FF_p$-schemes with a nilpotent kernel then the restriction gives rise to an equivalence of categories 
	${\sf Vect} ^{\sf perf}(Y)\simeq {\sf Vect} ^{\sf perf}(X)$.
\end{Lemma}

\subsection{Representation schemes}

In the proof of Theorem \ref{isomorphism-for-morphisms} we use existence of a certain moduli scheme of vector bundles that we recall in this subsection. Let us  fix a positive integer $r$, a projective morphism $f: X_S\to S$ of schemes of finite type over some fixed field and an $f$-very ample line bundle $\cO_{X}(1)$. Let us assume that the fibers of $f$ are geometrically connected and there exists a section $x_S: S\to X_S$ of the morphism $f$. The following theorem was proven by X. Sun in \cite[Theorem 2.3]{Sun2019}. We formulate only a special case that is more suitable for our applications.

\begin{Theorem}
	There exists an $S$-scheme $R(r, X_S, x_S)$ that representes the functor associating to any morphism $T\to S$
	the set of framed sheaves $(\cE, \beta)$, where:
	\begin{enumerate}
		\item $\cE$ is a $T$-flat vector bundle on $X_T:=X_S\times _S T$, such that for every point $t\in T$ the restriction $\cE_t$  of $\cE$ to $(X_T)_t$ is geometrically Gieseker semistable  and has numerically vanishing Chern classes; moreover the quotients of the Jordan--H\"older filtration of $\cE_t$ are locally free at $x_T(t)$.
		\item $\beta: x_T^*\cE\to\cO_T^{\oplus r}$ is an isomorphism.  
	\end{enumerate}
\end{Theorem}

In the following $R(r, X_S, x_S)$ is called  the \emph{representation scheme} and  $\beta$ is called a \emph{frame} of $\cE$ at $x_T$.

\begin{Remark}
The above version is slightly different to the one constructed in \cite{Sun2019} and then used in \cite{Biswas-Parameswaran-Kumar2025} and \cite{Sun-Zhang2025}. It corresponds to an open subscheme of the corresponding representation scheme for the Hilbert polynomial of the trivial bundle, classifying vector bundles (this is an open condition, e.g., by \cite[Lemma 2.1.8]{Huybrechts-Lehn2010}) with numerically vanishing Chern classes (which is also an open condition). This scheme is much better behaved when considering a morphism $f: X\to Y$ and the rational map on representation schemes defined by the pullback. In general, the author does not know a good reason for the behaviour of Hilbert polynomials as claimed in all the above cited papers, e.g.,  \cite[proof of Theorem 4.2]{Sun2019}, \cite[Proposition 2.1]{Biswas-Parameswaran-Kumar2025} and various places in \cite{Sun-Zhang2025}. In fact, this seems to cause problems even when dealing with the Verschiebung map defined by the Frobenius morphism and which appears all over these papers.
\end{Remark}

\subsection{Specialization homomorphism}

It is well-known that for a  flat proper morphism  with geometrically connected fibres the specialization homomorphism of \'etale fundamental groups need not be surjective in general (see \cite[Expos\'e X,  Remarques 2.5]{SGA1}).
However, the following lemma shows that this can happen only over a proper closed subset of the base.

\begin{Lemma}\label{surjectivity-on-fund-group-for-spreading-out}
Let $S$ be a locally Noetherian integral scheme and let $f: X\to S$ be a flat proper morphism  with geometrically connected fibres. Let $\eta$ be the generic point of $S$ and let $\bar \eta$ be a geometric point lying over $\eta$.
Then there exists a non-empty open subset $U\subset S$ such that for every point $s\in U$ and a geometric point $\bar s$
lying over $s$, the specialization homomorphism
${\sf{sp}}: \pi^{\sf{\acute et}}_1(X_{\bar \eta})\to\pi^{\sf{\acute et}}_1(X_{\bar s})$
is surjective.
\end{Lemma}

\begin{proof}
There exists a finite radicial extension $\kappa (\eta)\subset L\subset \kappa (\bar \eta)$ such that $(X_L)_{\sf{red} }$
is geometrically reduced (see \cite[Proposition 4.6.6]{EGA4-2}). Let $S'\to S$ be a normalization of $S$ in the extension
 $\kappa (\eta)\subset L$. Let $X'$ be the reduced scheme structure on $X\times _SS'$ and let $f': X'\to S'$ be the induced map.
 By \cite[Lemma 054Z]{StacksProject} the generic fiber of $f'$ is reduced and hence it is isomorphic to $(X_L)_{\sf{red} }$.
 So by \cite[Th\'eor\`eme 12.2.4]{EGA4-3} there exists a non-empty open subset $U'\subset S'$ such that the fibers of $f'|_{(f')^{-1}(U')}$ are geometrically reduced. By generic flatness (see \cite[Proposition 052B]{StacksProject}) we can also assume that $f'|_{(f')^{-1}(U')}$ is flat.  Then by \cite[Expos\'e X, Corollaire 2.4]{SGA1}
 for every geometric point $\bar s'$ lying over $s'\in U'$    the specialization 
 homomorphism
 ${\sf{sp}}: \pi^{\sf{\acute et}}_1(X'_{\bar \eta})\to\pi^{\sf{\acute et}}_1(X'_{{\bar s'}})$
 is surjective. But if $\bar s$  denotes the image of $\bar s'$ then
 by \cite[Lemma 0C0K]{StacksProject} we have a commutative diagram
$$
\xymatrix{\pi^{\sf{\acute et}}_1(X'_{\bar \eta })\ar[d]^{\sf sp}\ar[r]& \pi^{\sf{\acute et}}_1(X_{\bar \eta})\ar[d]^{\sf sp}\\
	\pi^{\sf{\acute et}}_1(X'_{{\bar s'}})\ar[r]&\pi^{\sf{\acute et}}_1(X_{\bar s})\\
}
$$
Since the horizontal maps in the above diagram are isomorphisms, it is sufficient to take as $U$ and non-empty open subset contained in the image of $U'\to S$.
\end{proof}

\section{$F$-divided bundles on fibrations}

\begin{Lemma}\label{factorization-of-proper-morphism}
	Let $f: X\to Y$ be a proper surjective map of connected Noetherian $F$-finite $\FF_p$-schemes. Assume that all geometric fibers of $f$ are connected. Then $f$ decomposes into a composition of $g:X\to Z$ with $g_*\cO_X=\cO_Z$ and a finite universal homeomorphism $h: Z\to Y$.
\end{Lemma}

    \begin{proof}
    By Stein factorization
    \cite[\href{https://stacks.math.columbia.edu/tag/03H0}{Tag
    03H0}]{StacksProject}, $f$ factors as $X \xrightarrow{g} Z \xrightarrow{h}
    Y$ with $h$ finite and $\mathcal{O}_Z \xrightarrow{\cong}
    g_*\mathcal{O}_X$. For any geometric point $\bar{y}$ of $Y$, since $g$ is
    surjective and $f^{-1}(\bar{y})$ is connected, the fiber $h^{-1}(\bar{y})$
    is connected. As it is also finite over $\kappa(\bar{y})$, the scheme
    $h^{-1}(\bar{y})$ must be the spectrum of an Artinian local $\kappa(\bar{y})$-algebra. This implies $h$ is universally injective
    (the geometric fiber of any base change of $h$ will remain so). Since $h$
    is universally closed and surjective, it is therefore
    a universal homeomorphism.
\end{proof}

The following theorem is an analogue of \cite[Theorem 1.3]{Bhatt-Scholze2017}. 

\begin{Theorem}\label{descent-for-proper-fibrations-old} Let $f: X\to Y$ be a proper surjective map of connected Noetherian $F$-finite $\FF_p$-schemes. Assume that all geometric fibers of $f$ are connected.
Then  the pullback functor $f^*\colon \Vect^\mathrm{perf}(Y) \to \Vect^\mathrm{perf}(X)$ is fully faithful. Its essential image consists precisely of those $F$-divided bundles $\EE$  on $X$ whose restriction to every geometric fiber $X_{\bar{y}}$ (for $\bar{y}$ a geometric point of $Y$) is trivial. 
\end{Theorem}

\begin{proof} 
	   By  Lemma \ref{passing-to-universal-homeomorphism} and Lemma \ref{descent-for-proper-fibrations-old} we can assume that $f_*\mathcal{O}_X = \mathcal{O}_Y$. Then the projection formula implies that the pullback functor $f^*: \Vect^{\perf}(Y) \to \Vect^{\perf}(X)$ is fully faithful. So we need to show that $\EE\in \Vect ^{\perf} (X)$ descends (necessarily uniquely) to $Y$ if and only if for all geometric points $\bar{y}$ of $ Y$, $\EE$ is trivial on the fiber $X_{{\bar{y}}}$.

For any $\FF \in \Vect^{\perf} (Y)$  the restriction of $\FF$ to any geometric point $\bar{y}$ is trivial
	as $ \Vect^{\perf} (\bar{y})=\Vect (\bar{y})$. Since the absolute Frobenius commutes with any morphisms, the restriction of $f^*\FF$ to $X_{\bar y}$ is also trivial, which proves one implication.	
	
To prove the other one, by Lemma \ref{lem:finiteness} 
	we can assume that both $X$ and $Y$ are reduced. Let us write $\EE= \{E_i, \sigma_i \}$ and assume that for all geometric points $\bar{y}$ of $Y$, $\EE$ is trivial of rank $r$ on the fiber $X_{\bar{y}}$. 
	
	Let us first assume that $Y=\Spec k$ for some field $k$. Let us set $M_i=H^0(X,E_i)$.
	By assumption we have $H^0(X, \cO_X)=k$ and then
	the base change $M_i\otimes _k\bar{k}\simeq H^0(X_{\bar{k} },\cO_{X_{\bar{k}}}^{\oplus r})$ implies that 
	$M_i$ is an $r$-dimensional vector space and the relative evaluation map  $f^*M_i\to E_i$ is an isomorphism. 
	On the other hand, we have $k$-linear injective maps
	$$M_{i+1}\hookrightarrow H^0(X,F_{X*}F_X^*E_{i+1})=F_{k*}H^0(X,F_X^*E_{i+1})\mathop{\longrightarrow}^{H^0(\sigma_i)}F_{k*}M_{i}$$
of $k$-vector spaces of the same dimension,	so all the maps must be isomorphisms. This shows that $M_{i+1}\simeq  F_{k*}M_{i}$ and for every $i$ we have a commutative diagram
	$$
	\xymatrix{
		F_X^*f^*M_{i+1}\ar[rd]^{\simeq}\ar@{=}[r]& f^*F_k^*M_{i+1}\ar[r]^{\simeq}&f^*M_i\ar[d]^{\simeq}\\
		&F_X^*E_{i+1}\ar[r]^{\simeq}&E_i 
	}$$
Therefore $\EE$ descends to $\Spec k$.
	
	Now let us consider the general situation. The isomorphism $\sigma _i$ induces $E_{i+1}\to F_{X*}E_i$ and
	hence $$f_*E_{i+1}\to f_*F_{X*}E_i=F_{Y*}f_*E_i.$$
	We will show that the adjoint maps $\tau_i: F_Y^*f_*E_{i+1}\to f_*E_i$ are isomorphisms so that we obtain an $F$-divided bundle $f_*\EE=\{ f_*E_i, \tau_i\}$ on $Y$.
	
	By the above for every point $y\in Y$ the  restriction of $\EE$ to $X_y$ descends to $y$.
	Let $\hat{X}_y$ be the formal completion of $X$ along $X_y$ and let $\imath: \hat{X}_y\to X$ and  $\hat{f}: \hat{X}_y\to {\sf Spf}\hat{\cO}_{Y,y}$ be the induced maps. Then $ \Vect^{\perf} ({\sf Spf}\hat{\cO}_{Y,y})\simeq  \Vect^{\perf} (y)$
	and  $ \Vect^{\perf} (\hat{X}_y)\simeq  \Vect^{\perf} (X_y)$ so  $\imath ^*\EE$ descends to ${\sf Spf}\hat{\cO}_{Y,y}$. This implies that we can define
	$$\hat{f}_*(\imath ^*\EE)=\{ \hat{f}_*(\imath ^*E_i), \hat{f}_*(\imath ^*\sigma_i)\}$$ 
	and the canonical relative evaluation map  $\hat{f}^*\hat{f}_*(\imath ^*\EE)\to \imath ^*\EE$ is an isomorphism.
	In more concrete terms, if $X_n=X\times_Y\Spec \cO_{Y,y}/m_y^n$ is the $n$-th infinitesimal neighbourhood of $X_y$ 
	and $E_{i,n}$ is the restriction of $E_i$ to $X_n$ then $\imath ^*E_i=\varprojlim E_{i,n}$,
$ \hat{f}_*(\imath ^*E_i)=\varprojlim H^0(X_n, E_{i,n})$ and the maps $\hat{f}_*(\imath ^*\sigma_i)$ are defined analogously. By the theorem on formal functions we have natural isomorphisms $(\hat{f_*E_i})_y\simeq \hat{f}_*(\imath ^*E_i)$ and $\hat{f}_*(\imath ^*\EE)$  gives an object of $\Vect^{\perf} (\hat{\cO}_{Y,y})$. 	 Since ${\cO}_{Y,y}\to \hat{\cO}_{Y,y}$ is faithfully flat this implies that all $f_*E_i$ are locally free at $y$ and the canonical maps 
	$$F_{\cO_{Y,y}}^*(f_*E_{i+1})_y\to (f_*E_{i})_y$$
	are isomorphisms. Since this holds for all $y\in Y$ the maps $\tau_i$ are isomorphisms
    and $f_*\EE$ is a well defined object of $\Vect^{\perf} (Y)$. We also have the induced relative evaluation map $f^*f_*\EE\to \EE$, which is an isomorphism on every formal fiber, so it is also an isomorphism.
\end{proof}

	The following example shows that unlike in the usual vector bundle case we cannot expect that if $\EE$ descends to $Y$ then it is trivial on all the fibers of $f$.

\begin{Example}
	Let $\eta$ be the generic point of $X=\AA^1_{\bar \FF_p}\backslash \{0\}$.
	By \cite[Proposition 3.4]{Langer-Zhang2025}  the restriction map $\Pic _F(X)\to \Pic _F(\eta)$ is injective, where $\Pic _F$ denotes the group of isomorphism classes of $F$-divided line bundles.
	On the other hand, by \cite[Proposition 3.4 and Corollary 3.7]{Kindler2015} the group $\Pic _F(X)$ is non-trivial.
	In fact, we have $H^0(X, \cO_X^{\times})/\bar \FF_p^{\times}\simeq \ZZ$ so $\Pic _F(X)\simeq \ZZ_p/\ZZ$.
	It follows that there exist non-trivial $F$-divided line bundles  in $\Vect^{\perf}(\eta)=\Vect^{\perf}(\FF_p(z))$.
	However, $\Vect^\perf(\overline{\FF_p(z)})=\Vect(\overline{\FF_p(z)})$ so the base change of such an $F$-divided line bundle to 
	$\overline{\FF_p(z)}$ is the unit object. 
\end{Example}

\begin{Corollary} \label{descent-for-proper-fibrations} 
Let $f: X\to Y$ be a proper morphism of normal integral schemes of finite type over a perfect field $k$ of positive characteristic and let $\EE\in \Vect ^{\perf} (X)$.  Then $\EE$ descends (necessarily uniquely) to $Y$ if and only if for all closed points $y\in Y$, $\EE$ is trivial on the reduced scheme structure $( X_y)_{\sf red}$ of the fiber $X_{y}$.
\end{Corollary}

\begin{proof}
One implication is clear as there are no no-trivial $F$-divided bundles on the spectrum of a perfect field. To prove the other one let us assume that  for all closed points $y\in Y$, $\EE$ is trivial on $( X_y)_{\sf red}$. Then by Lemma \ref{lem:finiteness} it is trivial also on the fiber $X_{y}$. As in the proof of Theorem \ref{descent-for-proper-fibrations-old} this implies that all $ f_*E_i$ are locally free at $y$. But in our case, the closed points are dense in any closed subset of $Y$. Since local freeness is an open condition, this implies that all $ f_*E_i$ are locally free on $Y$. Similarly, the adjoint maps $\tau_i: F_Y^*f_*E_{i+1}\to f_*E_i$ are isomorphisms at all closed points of $Y$ so they are isomorphisms. This implies that $\EE$ descends to the $F$-divided bundle $\{ f_*E_i, \tau _i\}$.
\end{proof}

\begin{Corollary}\label{proper-fibration-relative-gerbe}
	Let $f: X\to Y$ be a proper surjective map of connected Noetherian $F$-finite $\FF_p$-schemes. Assume that all geometric fibers of $f$ are connected. Then  the induced morphism $f_*: \Pi^{F\sf{\textendash
			div}}_{X}\to\Pi^{F\sf{\textendash div}}_{Y} $ is a relative gerbe over $\cO_Y(Y)^{\sf perf}=\cO_X(X)^{\sf perf}$. 
\end{Corollary}

\begin{proof}
  By \cite[Lemma 1.6]{Langer-Zhang2025} it is sufficient to show that for any $\EE\in\Vect^\perf(Y)$, 
  any subobject $\FF\subseteq f^*\EE$ is contained in the essential image of $f^*$.
    Thanks to Lemma \ref{factorization-of-proper-morphism}, it is enough to show that $\FF|_{X_{\bar{y}}}$ is trivial for all geometric points $\bar{y}$ of $Y$. Since $\Vect^\perf(X_{\bar{y}})$ is a $\kappa(\bar{y})$-Tannakian category (cf.~Theorem \ref{F-divided-sheaves-form-Tannakian-category}) and $\FF|_{X_{\bar{y}}}$ is a subobject of the trivial object $\EE|_{X_{\bar{y}}}$, it must be trivial.   
\end{proof}

{

	\section{The relative Gieseker's problem}

In this section $k$ is an algebraically closed field of positive characteristic.

\subsection{Genuinely ramified morphisms}

When considering a finite  morphism  $f: X\to Y$  of normal projective $k$-varieties, we  fix an ample line bundle $\cO_Y(1)$ on $Y$ and we always consider (slope or Gieseker) (semi)stability on $Y$ with respect to $\cO_Y(1)$ and  (semi)stability on $X$ with respect to $f^*\cO_Y(1)$.
	
The following definition was first introduced in \cite[Definition 3.1]{Parameswaran-Subramanian2010} in the curve case and then generalized to higher dimensional smooth varieties in \cite{Biswas-Parameswaran-Kumar2025}. In the definition we allow normal varieties but we restrict to regular parts.

Let $f: X\to Y$ be a finite surjective morphism of normal projective $k$-varieties.
Then there exists a big open subset $V\subset Y_{\sf reg}$ with $U:=f^{-1}(V)\subset X_{\sf reg}$. Since we have canonical isomorphisms  $\pi^{\sf{\acute et}}_1(U)\simeq \pi^{\sf{\acute et}}_1 (X_{\sf reg})$ and $\pi^{\sf{\acute et}}_1(V)\simeq \pi^{\sf{\acute et}}_1 (Y_{\sf reg})$, the restriction  $f|_U: U\to V$ defines a canonical homomorphism $f_*: \pi^{\sf{\acute et}}_1(X_{\sf reg})\to\pi^{\sf{\acute et}}_1(Y_{\sf reg}) $.

\begin{Definition}
We say that $f$ is \emph{regularly genuinely ramified} if  the field extension $K(X)/K(Y)$ is separable (or, equivalently, $f$ is generically smooth) and the induced homomorphism $f_*: \pi^{\sf{\acute et}}_1(X_{\sf reg})\to\pi^{\sf{\acute et}}_1(Y_{\sf reg}) $ is surjective. 
\end{Definition}

	The proof of the following lemma is the same as that of  \cite[Lemma 6.4]{Biswas-Parameswaran-Kumar2025} and we skip it.
	
	\begin{Lemma}
		Let $f: X\to Y$ be a regularly genuinely ramified morphism of normal projective $k$-varieties.  Assume that the field extension $K(X)/K(Y)$ is Galois of degree $d$.
		Then there exists a big open subset $V\subset Y_{\sf reg}$ with $U:=f^{-1}(V)\subset X_{\sf reg}$
		and  line bundles $L_{j}\subsetneq \cO_{U}$  for $j=1,...,d-1$ such that
		$$ (f|_V)^*((f|_U)_*\cO_U/\cO_V) \subset \bigoplus _{j=1}^{d-1} L_j.$$ 
	\end{Lemma}

As a corollary (see  \cite[Lemma 6.5]{Biswas-Parameswaran-Kumar2025} for a slightly weaker version for Gieseker semistable vector bundles in higher dimensions) we obtain the following result:
	
	\begin{Corollary}
		Let $f: X\to Y$ be a regularly genuinely ramified  morphism of normal projective $k$-varieties. Then for any slope semistable reflexive sheaf $E$  on $Y$ we have
		$$\mu _{\sf max} ((E\otimes (f_*\cO_X/\cO_Y) )^{**})<\mu (E).$$
	\end{Corollary}
	
In the following if $E$ is a reflexive sheaf on $Y$ then we write $f^{[*]}E$ to denote the reflexivization of $f^*E$.
The above corollary implies the following corollary (see \cite[Lemma 4.3]{Biswas-Parameswaran2022} for the curve case and \cite[Proposition 2.3]{Biswas-Parameswaran-Kumar2025} for vector bundles on smooth higher dimensional varieties).

	\begin{Corollary} \label{fully-faithful-for-finite-morphism}
		Let $f: X\to Y$ be a regularly genuinely ramified  morphism of normal projective $k$-varieties. Then for any slope semistable reflexive sheaves $E_1, E_2$  on $Y$ with the same slope the natural map
		$$\Hom _{Y} (E_1, E_2)\to \Hom _{X} (f^{[*]}E_1, f^{[*]}E_2)$$	
		is an isomorphism.	
	\end{Corollary}
	
	\begin{proof}
		There exists a big open subset $V\subset Y_{\sf reg}$ such that $U:=f^{-1}(V)\subset X_{\sf reg}$ and
		restrictions of $E_1$ and $E_2$ to $V$ are locally free. Then the same arguments as that  in \cite[Proposition 2.3]{Biswas-Parameswaran-Kumar2025} show that 
		$$\Hom _{V} (E_1|_V, E_2|_V)\to \Hom _{U} ((f|_U)^*(E_1|_V), (f|_U)^*(E_2|_V))$$	
		is an isomorphism, which implies the required assertion.
	\end{proof}
	
Using the above facts, the following theorem can be proven by adapting the arguments from the proof of  \cite[Theorem 1.1]{Biswas-Parameswaran-Kumar2025}, which treats the case of smooth varieties. We omit the proof. It should be noted, however, that the proof of \cite[Theorem 1.1]{Biswas-Parameswaran-Kumar2025} contains an error in the claim that one can spread out infinitely many bundles simultaneously. This issue can be easily corrected by constructing a model for the Zariski closure of the framed bundles in the representation scheme (see an application of this technique below in the proof of Theorem \ref{isomorphism-for-morphisms}). 

\begin{Theorem}\label{faithful-flatness-for-finite-morphisms}
	Let $f: X\to Y$ be a finite, generically smooth morphism of normal projective $k$-varieties.  If $Y$ is smooth and	$f_*: \pi^{\sf{\acute et}}_1(X)\to\pi^{\sf{\acute et}}_1(Y) $ is an isomorphism then
	$f_*: \pi^{F\sf{\textendash	div}}_1(X)\to\pi^{F\sf{\textendash div}}_1(Y) $ is an isomorphism.
\end{Theorem}	

Unfortunately, \cite[Proposition 4.1]{Biswas-Parameswaran-Kumar2025} that is used in the above theorem uses the Zariski--Nagata purity theorem and hence it can be generalized only in case $Y$ is smooth (see \cite[Lemma 3.2]{Sun-Zhang2025}). The use of \cite[Lemma 3.2]{Sun-Zhang2025} is the main reason why we add assumption that $Y$ is smooth in the above theorem. Another reason for assuming that $Y$ is smooth comes from the above lemmas. Note also that one cannot easily prove that in the above setup if  $Y$ is smooth and	$f_*: \pi^{\sf{\acute et}}_1(X_{\sf reg})\to\pi^{\sf{\acute et}}_1(Y) $ is an isomorphism then
$f_*: \pi^{F\sf{\textendash	div}}_1(X_{\sf reg})\to\pi^{F\sf{\textendash div}}_1(Y) $ is an isomorphism. In this case one runs into problems  with a specialization homomorphism on the \'etale fundamental group for the non-proper variety $X_{\sf reg}$ (cf. \cite[Section 4]{Esnault-Srinivas2016}).

\medskip

In the next section we will also need the following analogue of \cite[Theorem 4.4]{Biswas-Parameswaran2022}.
Since the proof, as in case of Corollary \ref{fully-faithful-for-finite-morphism} is very 
similar to that from the curve case, we omit it (see also \cite[Theorem 1.1]{Patel-Weissmann2026}).

\begin{Lemma}\label{pullback-of-stable-is-stable}
	Let $f: X\to Y$ be a regularly genuinely ramified finite morphism of normal projective $k$-varieties. Then for any slope stable reflexive sheaf $E$  on $Y$ the reflexivized pullback $f^{[*]}E$ is also slope stable.
\end{Lemma}

\subsection{The general case}

In the proofs below we use the following important generalization of \cite[Proposition 2.3]{Esnault-Mehta2010} that is contained in \cite[Proposition 6.3]{Langer-Zhang2025}.		

\begin{Proposition}\label{Jordan-Holder-for-F-divided-bundles}
	Let $X$ be a connected normal projective variety defined over a field $k$ of positive characteristic and let us fix an ample line bundle on $X$. For any $\EE\in \Vect^{\perf}(X)$ there exists $n_0\in \ZZ_{\ge 0}$ such that $\EE$ is a successive extension of $F$-divided bundles $\FF\in \Vect^{\perf}(X) $  with the property that all $F_n$ with $n\ge n_0$ are slope stable with numerically vanishing Chern classes.
\end{Proposition}

	\begin{Theorem}\label{faithful-flatness-for-morphisms}
		Let $f: X\to Y$ be a surjective morphism of normal projective $k$-varieties.  If $Y$ is smooth and	$f_*: \pi^{\sf{\acute et}}_1(X)\to\pi^{\sf{\acute et}}_1(Y) $ is surjective then
		$f_*: \pi^{F\sf{\textendash	div}}_1(X)\to\pi^{F\sf{\textendash div}}_1(Y) $ is faithfully flat.
	\end{Theorem}	
	
	\begin{proof}
		The morphism $f$ is proper so we can consider its Stein factorization (see \cite[\href{https://stacks.math.columbia.edu/tag/03H0}{Tag	03H0}]{StacksProject}) into $X \xrightarrow{g} Z \xrightarrow{h} Y$ with $h$ finite and $\mathcal{O}_Z \xrightarrow{\cong}
		g_*\mathcal{O}_X$. Let us also note that $g$ is proper with geometrically connected fibres and $Z$ is normal, integral and projective. So $g_*: \pi^{F\sf{\textendash	div}}_1(X)\to\pi^{F\sf{\textendash div}}_1(Z) $  is faithfully flat by Corollary \ref{proper-fibration-relative-gerbe}.
		Therefore in the following we can assume that $f$ is finite. Then we can find an intermediate field extension $K(X)/L/K(Y)$ so that  $K(X)/L$
		is  purely inseparable and $L/K(Y)$ is separable. Taking as $Z$ the normalization of $Y$ in $L$ we 
		obtain finite morphisms $g: X\to Z$ and $h: Z\to Y$ such that  $g$ is a finite universal homeomorphism and $h$ is  finite generically smooth. By Lemma \ref{passing-to-universal-homeomorphism} $g$
		induces an isomorphism $\pi^{F\sf{\textendash	div}}_1(X)\to\pi^{F\sf{\textendash div}}_1(Z)$ so we can assume that $f$ is generically smooth. 

Let us first note that the pullback functor $f^*: \Vect^{\perf}(Y) \to \Vect^{\perf}(X)$ is fully faithful, i.e.,
		for any $\EE_1, \EE_2 \in \Vect^{\perf} (Y)$ the canonical map  
		$$\Hom (\EE_1, \EE_2)\to \Hom  (f^{*}\EE_1, f^*\EE_2)$$	
		is an isomorphism. Indeed, by  Proposition \ref{Jordan-Holder-for-F-divided-bundles} we can assume  that for all $j$
		both $E_{1,j}$ and $E_{2,j}$ are slope (or even Gieseker) semistable. So the required assertion follows from the fact that for any slope semistable vector bundles $E_1$ and $E_2$ the canonical map 
		$$\Hom _{Y} (E_1, E_2)\to \Hom _{X} (f^*E_1, f^{*}E_2)$$	
	is an isomorphism. This last isomorphism follows from Corollary \ref{fully-faithful-for-finite-morphism}
as 	 $ \pi^{\sf{\acute et}}_1(X_{\sf reg})\to\pi^{\sf{\acute et}}_1(X) $ is surjective, and hence $f$ is regularly genuinely ramified.  

		By \cite[Proposition 2.21]{Deligne-Milne1982} to finish the proof it is sufficient to show that for any $\EE\in\Vect^\perf(Y)$ and any subobject $\GG\subseteq f^*\EE$, $\GG$ is contained in the essential image of $f^*$. By Proposition \ref{Jordan-Holder-for-F-divided-bundles} we can assume  that  $\EE$ admits a filtration by subobjects so that the quotients $\FF_1,..., \FF_m$ are $F$-divided bundles on $Y$  with the property that all $F_{j,n}$ are slope stable with numerically vanishing Chern classes.
		Then by Lemma \ref{pullback-of-stable-is-stable}   $f^*\EE$ is a successive extension of the $F$-divided bundles $f^*\FF_j\in \Vect^{\perf}(X) $  with all $f^*F_{j,n}$ slope stable with numerically vanishing Chern classes. 
		In particular, all $f^*\FF_j$ are simple objects of $\Vect^{\perf}(X)$.  Since the category $ \Vect^{\perf}(X) $ is Tannakian, an easy induction on the rank of $\GG$ shows that $\GG$ is an extension of some of the $F$-divided bundles  $f^*\FF_1,..., f^*\FF_m$ and hence it is contained in  the essential image of $f^*$.
	\end{proof}

	\begin{Remark}
		In case $f$ is a finite genuinely ramified morphism between smooth projective $k$-varieties the above theorem was claimed in \cite[Remark 5.1]{Biswas-Parameswaran-Kumar2025}, who checked that in this case  $f^*: \Vect^{\perf}(Y) \to \Vect^{\perf}(X)$ is fully faithful. It is not clear to the author how they planned to prove the second condition needed for application of  \cite[Proposition 2.21]{Deligne-Milne1982}.
	\end{Remark}

\begin{Remark}
If $f_*: \pi^{\sf{\acute et}}_1(X)\to\pi^{\sf{\acute et}}_1(Y) $ is surjective then $f^*: \Vect^{\perf}(Y) \to \Vect^{\perf}(X)$ is a fully faithful exact functor. However, the essential image of $f^*$ does not need to be a Serre subcategory. For example, if $X$ is an ordinary elliptic curve and $f$ is the map of $X$ to a point
then the essential image of $f^*$ consists of trivial objects but an extension of $\unit_X$ by $\unit_X$ need not be trivial as $\Ext ^1(\unit_X, \unit_X)=H^1(X,\cO_X)\ne 0$ (see, e.g., the proof of \cite[Theorem 15]{dos_Santos2007}).
\end{Remark}

\begin{Theorem}\label{isomorphism-for-morphisms}
Let $f: X\to Y$ be a surjective morphism of normal projective $k$-varieties with  $f_*\cO_X=\cO_Y$. If 	$f_*: \pi^{\sf{\acute et}}_1(X)\to\pi^{\sf{\acute et}}_1(Y) $ is an isomorphism then $f_*: \pi^{F\sf{\textendash	div}}_1(X)\to\pi^{F\sf{\textendash div}}_1(Y) $ is an isomorphism.
\end{Theorem}	
	
\begin{proof}
The condition $f_*\cO_X=\cO_Y$ implies that $f$ has geometrically connected fibres. By Theorem \ref{faithful-flatness-for-morphisms} it is sufficient to show that $f^*: \Vect^{\perf}(Y) \to \Vect^{\perf}(X)$ is essentially surjective. By Corollary \ref{descent-for-proper-fibrations} it is sufficient to show that for any object $\EE$ of $\Vect^{\perf}(X)$ and any closed point $y\in Y$ the  restriction $\EE_y$ of $\EE$ to the reduction $(X_y)_{\sf red}$ of the  fiber $X_{y}$ is trivial. 

In the following we fix some ample line bundle $\cO_{\tilde X}(1)$. By Proposition \ref{Jordan-Holder-for-F-divided-bundles} there exists some $n_0$ such that all $E_n$ with $n\ge n_0$ are Gieseker semistable with numerically vanishing Chern classes. Without loss of generality we can assume that $n_0=0$. Let us fix a $k$-point $x$ of $X$ and consider the representation scheme $R(r, X,  x)/k$.
Let us fix a frame $\beta_0: x^*E_0\simeq \cO_k^{\oplus r}$. This uniquely determines the frames $\beta_n:  x^*E_n\simeq \cO_k^{\oplus r}$ such that $F_k^*\beta_{i+1}=\beta_i$. Let us define the locus $A_j\subset R(r, X)$ as the Zariski closure of the set $\{ (E_n, \beta_n)\} _{n\ge j}\subset R(r,  X) (k)$.

 Since $A_{j+1}\subset A_j$, the sequence $\{ A_j\}_{j\ge 0}$ stabilizes and for large $n$ we have $A_n=A:=\bigcap A_j$. Let $R_0$ be the open subset of $ R(r, X,x)$ corresponding to pairs $[(E, \beta)]$ such that $F_{X}^*E$ is Gieseker semistable. Then pullback by the absolute Frobenius morphism on $X$ defines a morphism $R_0\to  R(r,  X)$. After restricting to $A\cap R_0$ (which is non-empty as it contains all $[(E_n, \beta_n)]$ for sufficiently large $n$) we have a well defined morphism, which gives a dominant rational map $A\dashrightarrow A$. This morphism is not $k$-linear and to make it $k$-linear we need to consider the relative Frobenius morphism $F_{ X/k}:  X\to X'$. Then we have 
$ R(r, X', x')\stackrel{\simeq}{\longrightarrow}  R(r,  X,x)\times _{F_k}k$, where $x'=F_{X/k} (x)$.  
If  $A'$ is the reduced preimage of $A\times _{F_k}k$ then we obtain a dominant rational $k$-morphism $\varphi: A'\dashrightarrow A$, which is the restriction of an analogue of the classical Verschiebung rational map.

Now we need to spread out the whole situation. There exists a finitely generated $\FF_p$-algebra $R\subset k$ and a  scheme $X_S$ of finite type over $S=\Spec R$ such that $X/k$ is isomorphic to the generic geometric fiber of $X_S\to S$.  Shrinking $S$ if necessary we can assume that $X_S\to S$ is proper and flat with geometrically connected fibres. We can also assume that all fibers of $X_S\to S$ are geometrically normal and geometrically integral. 
Similarly, we can find $f_S: X_S\to Y_S$ with $X$ and $Y$ projective over $S$ so that $f_S$ is isomorphic to $f$ over the generic geometric point of $S$. By Zariski's main theorem for all $s\in S$ the morphisms  $f_s:  X_s\to Y_s$ are birational and satisfy $f_{s*}\cO _{ X_s}=\cO _{Y_s}$. We also need a section $x_S: S\to X_S$ which gives the point $x$ at the geometric generix point of $S$.

Then we construct $S$-flat models $A_S\subset R(r, X_S, x_S)$ for $A$ and $A_S'\subset R(r, X'_S, x'_S)$ for $A'$.
We have a dominant rational map of $S$-schemes $\varphi_S:A'_S\dashrightarrow A_S$ extending $\varphi$ and
defined by pullback via the relative Frobenius morphism  $F_{X_S/S}:X_S\to X'_S$.
Shrinking $S$ if necessary we can assume that the restriction $\varphi_s: A'_s\dashrightarrow A_s$ is a dominant rational map for all closed points $s$ of $S$. For such $s$ we let  $\varphi_s^{(i)}$ denote the $i$th Frobenius twist of $\varphi_s$ and set $m_s=(\kappa (s):\FF_p)$. Since $\varphi _s$ is defined by $[(E, \beta )]\to [F_{X_s/s}^*E, F_{\kappa (s)}^*\beta]$, the composition $\varphi_s^{(m_s-1)}\circ ...\circ \varphi_s^{(1)}\circ \varphi _s$ defines a rational endomorphism $A_s\dashrightarrow A_s$ for any closed point $s\in S$. Now (as in \cite{Esnault-Mehta2010}) Hrushovski's theorem \cite[Corollary 1.2]{Hrushovski2004} implies that set of closed points of $A_s$, which are periodic for this rational endomorphism is dense in $A_s$. Such points correspond to framed pairs $[(G, \beta_G)]$, where $G$ is a geometrically Gieseker stable vector bundle on $ X_s$ such that for some $m\ge 1$ we have an isomorphism $(F^m_{ X_s})^*G\simeq G$. If $\bar s$ is a geometric point lying over $s$ then $G_{\bar s}$ on $X_{\bar s}$ gives by  \cite[Proposition 4.1.1]{Katz1972} a representation
of $\pi^{\sf{\acute et}}_1(X_{\bar s})$.

Let us consider $X\times _YX$ with two projections $p_{1}$ and $p_{2}$ onto $X$. Similarly, let us define $p_{1\bar s}$ and $p_{2\bar s}$. By \cite[Expos\'e IX, Corollaire 5.6]{SGA1}
and \cite[Expos\'e IX, Th\'eor\`eme 4.12]{SGA1} and our assumption, it follows that the two homomorphisms
$p_{1*}, p_{2*}: \pi^{\sf{\acute et}}_1(X\times _YX)\to\pi^{\sf{\acute et}}_1(X)$ coincide. By \cite[Lemma 0C0K]{StacksProject} we have a commutative diagram
$$
\xymatrix{\pi^{\sf{\acute et}}_1(X\times _YX)\ar[d]^{\sf sp}\ar@<-.5ex>[r]_-{p_2*} \ar@<.5ex>[r]^-{p_{1*}}& \pi^{\sf{\acute et}}_1(X)\ar[d]^{\sf sp}\\
\pi^{\sf{\acute et}}_1(X_{\bar s}\times _{Y_{\bar s}}X_{\bar s})\ar@<-.5ex>[r]_-{p_{2\bar s*}} \ar@<.5ex>[r]^-{p_{1\bar s*}}&\pi^{\sf{\acute et}}_1(X_{\bar s})\\
}
$$
By Lemma \ref{surjectivity-on-fund-group-for-spreading-out}, shrinking $S$ if necessary, we can assume that the vertical specialization maps are surjective and hence we have $p_{1\bar s*}=p_{2\bar s*}$. But then \cite[Expos\'e IX, Corollaire 5.6]{SGA1} and \cite[Expos\'e IX, Th\'eor\`eme 4.12]{SGA1} imply that
the map $f_*: \pi^{\sf{\acute et}}_1(X_{\bar s})\to\pi^{\sf{\acute et}}_1(Y_{\bar s}) $ is an isomorphism. 

In this way, $G_{\bar s}$ induces a representation of $\pi^{\sf{\acute et}}_1(Y_{\bar s})$, which by  \cite[Proposition 4.1.1]{Katz1972} gives rise to a vector bundle $G'$ on $Y_{\bar s}$ such that  $G_{\bar s}\simeq f_{\bar s}^* G'$. Since the framing $\beta_{G_{\bar s}}:  x_{\bar s}^*G_{\bar s}\simeq \cO_{\kappa (\bar s)}^{\oplus r}$ descends to a framing of $G'$ at $f_{\bar s}(x_{\bar s})$, the rational map $f_s^*: (f_s^*)^{-1}(A_s)\dashrightarrow A_s$, given by the pullback $[(E, \beta )]\to [(f^*_sE, f_s^*\beta)]$, is dominant.

Let us set $y_S=f_S(x_S)$ and consider the rational $S$-map $f_S^*: R(r, Y_S, y_S)\dashrightarrow R(r, X_S, x_S)$ given by the pullback by $f_S$. The above argument shows that the rational map $(f_S^*)^{-1}(A_S)\dashrightarrow A_S$ is dominant as it is dominant over all closed points $s\in S$. It follows that the rational map $(f^*)^{-1}(A)\dashrightarrow A$ over the geometric generic point of $S$ is also dominant. Let $B\subset A$ be a dense open subset contained in the image of the morphism on which this rational map is defined.

Now let us recall that the representation scheme $R(r, X, x)$ comes equipped with a universal family $\cU$ together with a universal framing $\beta_{\cU}$. This is a flat vector bundle on $R(r, X, x)\times _k  X\to R(r, X, x)$, which is geometrically Gieseker semistable on the fibers and such that for any locally Noetherian $k$-scheme $T$ and a $T$-flat family $E_T$ of geometrically Gieseker stable rank $r$ vector bundles with numerically vanishing Chern classes on the fibers of $f_T: T\times X\to T$ with framing $\beta_T$ along $x_T$, if $\varphi _{(E_T, \beta _T)}: T\to R(r, X, x)$ 
denotes the classifying morphism then $E_T\simeq (\varphi _{E_T} \times _k \id _{X})^*\cU$ (and similarly for the framing).

For each $k$-point ${y}$ of $ Y$ and each $k$-point $b$ of $B$ the restriction $\cU|_{\{ b\} \times (X_y)_{\sf red}}$
is trivial as $b=[(f^*E, f^*\beta)]$ for some $[(E, \beta )]\in R(r, Y, f(x))(k)$.
 So by the semicontinuity theorem  for every $a\in A (k)$ we have
$$\dim H^0((X_y)_{\sf red},\cU|_{\{ a\} \times (X_y)_{\sf red}} )\ge r.$$
Let us set  $\EE_y:= \EE|_{(X_y)_{\sf red}}$ and $E_{y, i}:=E_i|_{(X_y)_{\sf red}}$.  By construction and by the above we have
$$\dim H^0((X_y)_{\sf red}, E_{y,i} )\ge r$$
for all $i\ge 0$. Note also that the maps $H^0(E_{y, i+1})\to H^0(E_{y, i})$ are injective as they are induced from restrictions of the injective $p$-linear maps $E_{i+1}\to E_i$. 
This shows that $\EE _y^h ((X_y)_{\sf red})=\varprojlim  H^0(E_{y, i})$ has dimension $\ge r$ and hence $\EE_y$ is trivial by Lemma \ref{evaulation-map}. By Corollary \ref{descent-for-proper-fibrations} this implies that $\EE$ is 
contained in the essential image of $f^*: \Vect^{\perf}(Y) \to \Vect^{\perf}(X)$.		
\end{proof}}

\medskip

\begin{Corollary}\label{cor:isomorphism-for-morphisms}
		Let $f: X\to Y$ be a surjective morphism of normal projective $k$-varieties.   If $Y$ is smooth and	$f_*: \pi^{\sf{\acute et}}_1(X)\to\pi^{\sf{\acute et}}_1(Y) $ is an isomorphism then
		$f_*: \pi^{F\sf{\textendash	div}}_1(X)\to\pi^{F\sf{\textendash div}}_1(Y) $ is an isomorphism.
\end{Corollary}

\begin{proof}
As in the proof of Theorem \ref{faithful-flatness-for-morphisms} we can divide the proof into two cases: one when $f$ satisfies $f_*\cO_X=\cO_Y$	 and the second one when $f$ is finite. The first case follows from Theorem \ref{isomorphism-for-morphisms}.
In the second case as in the proof of Theorem \ref{faithful-flatness-for-morphisms}  we can reduce to the case when $f$ is genuinely ramified. In this case the assertion follows from Theorem \ref{faithful-flatness-for-finite-morphisms}.
\end{proof}

\medskip

\section*{Acknowledgements}

The  author was partially supported by Polish National Centre (NCN) contract number 2021/41/B/ST1/03741.
The author would like to thank Lei Zhang for useful conversations and previous colaboration on the related project
\cite{Langer-Zhang2025}. He would also like to thank Dario Wei{\ss}mann for pointing out some errors in the first version of the paper.

	\bibliographystyle{amsalpha}
	\bibliography{References}

\end{document}